\documentclass{amsart}
\usepackage{amsthm}
\usepackage{xcolor}
\usepackage{enumerate}
\usepackage{amssymb,latexsym,amsxtra,amscd,amsfonts}
\usepackage[mathscr]{eucal}
\usepackage{dsfont}
\usepackage{calrsfs}

\numberwithin{equation}{section}

\newtheorem{theorem}{Theorem}
\newtheorem{corollary}[theorem]{Corollary}
\newtheorem{lemma}[theorem]{Lemma}
\newtheorem{proposition}[theorem]{Proposition}

\newtheorem{example}{Example}

\makeatletter
\def\@biblabel#1{#1.}
\makeatother
\newcounter{obr}[section]

\newcounter{pvv}[section]
\renewcommand{\thepvv}{\thesection.\arabic{pvv}}
\newenvironment{pv}[2][]{\begin{trivlist}\refstepcounter{pvv}%
\item[\hspace{\labelsep}\normalfont\bfseries\thepvv. #2%
  \def\tmp{#1}\ifx\tmp\empty\else{} (#1)\fi.]}%
{\end{trivlist}}

\newcommand{\set}[2]{\ensuremath{\{ #1\,:\; #2\}}}

\newcommand{\nn}{\ensuremath{\mathbb N}}

\newcommand{\cc}{\ensuremath{\mathbb C}}
\newcommand{\rr}{\ensuremath{\mathbb{R}}}

\newcommand{\8}{\ensuremath{\infty}}

\def\tr{\mathop{\rm tr}\nolimits}

\newcommand{\f}{\ensuremath{\varphi}}
 
\renewcommand{\l}{\ensuremath{\lambda}}
\newcommand{\ro}{\ensuremath{\varrho}}

\newcommand{\si}{\ensuremath{\sigma}}
\newcommand{\al}{\ensuremath{\alpha}}

\newcommand{\vNa}{von Neumann algebra}

\newcommand{\Ca}{$C^\ast$-algebra}
\newcommand{\Csa}{$C^\ast$-subalgebra}

\newcommand{\HS}{Hilbert space}

\newcommand{\st}{such that}

\newcommand{\Cl}[1]{\ensuremath{{\mathcal #1}}}
\newcommand{\nor}[1]{\ensuremath{\|#1\|}}

\begin{document}

\begin{center}
{\bf \LARGE  The order topology on duals of $C^\ast$-algebras and von Neumann algebras}\\[1cm]
\end{center}

\begin{center}

EMMANUEL  CHETCUTI and JAN HAMHALTER\\
\end{center}

\begin{center}
Department of Mathematics\\
Faculty of Science\\
University of Malta \\
Msida, Malta\\
emmanuel.chetcuti@um.edu.mt \\
\end{center}

\begin{center}
Department of Mathematics\\
Faculty of Electrical Engineering\\
Czech Technical University in Prague \\
Technicka 2, 166 27 Prague 6, Czech Republic\\
hamhalte@math.feld.cvut.cz \\
\end{center}
\vspace{2cm}

{\small Abstract:  For a von Neumann algebra $\Cl M$ we study the order topology associated to the hermitian part $\Cl M_*^s$ and to intervals of  the predual $\Cl M_*$.  It is shown that  the order topology on $\Cl M_*^s$ coincides with the topology induced by the norm.  In contrast to this, it is proved that the condition of having the order topology associated to the interval $[0,\f]$  equal to that induced by the norm for every $\f\in \Cl M_*^+$, is necessary and sufficient for the commutativity of $\Cl M$.  It is also proved that if $\f$ is a positive bounded linear functional on a \Ca{} \Cl A{}, then the norm-null sequences in $[0,\f]$ coincide with the null sequences with respect to the order topology on $[0,\f]$ if and only if the \vNa{} $\pi_\f(\Cl A)'$  is of finite type  (where $\pi_\f$ denotes  the corresponding GNS representation). This fact allows us to give a new topological characterization of  finite \vNa s.   Moreover, we demonstrate that convergence to zero for norm and order topology on order-bounded parts of dual spaces are nonequivalent for all \Ca s that are not of Type $I$.




\section{Introduction and preliminaries}
The aim of the paper is to present new results on the interplay between the norm topology and the order topology on dual spaces of \Ca s{} and \vNa s.

 The following variant of the squeezing lemma is well known to every student of calculus. Let $(y_n)$ be   an increasing sequence  of real numbers  and $(z_n)$ be a decreasing  sequence  of real numbers such that $y_n\le z_n$ for all $n$. Suppose that  the supremum of $(y_n)$ is the same as the infimum of $(z_n)$ and equals to $\al$. Then, for any other sequence, $(x_n)$, squeezed between $(y_n)$ and $(z_n)$ (i.e. $y_n\le x_n \le z_n$ for all $n$)  we have that $(x_n)$ converges to $\al$.  This simple principle  lies behind the definition  of order convergence, and subsequently,  of the order topology $\tau_o(P)$  on a general poset $P$ which is defined as the finest topology on $P$ for which the squeezing lemma holds. Order convergence has been studied in the context of posets and lattices by various authors  \cite{Birkhoff1,Birkhoff2, Kantorovich}  (see also \cite{Frink,MaAn,McShane}).  If the order structure of $P$ is  functional analytic it may appear that the order topology captures some analytic features of $P$. We mention, for example, the beautiful result of  Floyd and  Klee \cite{FloydKlee} saying that a normed space $X$ is reflexive if and only if the order  topology on the lattice of all closed subspaces of $X$ ordered by set-theoretic inclusion is Hausdorff. Continuing this line of research, the order topology on the lattice of closed subspaces of a \HS{} was studied in \cite{Palko,CHW}. A systematic treatment of various aspects of order topologies on structures associated with \vNa s was carried out in \cite{CHW}. In this case the order topology comes from the standard  operator order defined on the self-adjoint operators acting on a \HS{} $H$; namely, for self-adjoint operators $T$ and $S$, we have  $T\le S$  if the operator $S-T$ has a positive spectrum.   The work in \cite{CHW} showed that the order topology on  the self-adjoint part of a \vNa{},  albeit being far from being a  linear topology, on bounded parts it coincides -- in a surprising way --   with the strong operator topology.  Besides,  finite von Neumann algebras were characterized in terms of the order topology on its projection lattice.  Study  of order topologies on \vNa s that are induced by the star order was given in an interesting  recent work \cite{Bohata}.

In    the present paper we  study  the order topology induced by the canonical  order on duals and preduals of \Ca s.  The order  is defined as follows: For linear functionals $\f$ and $\psi$ on a \Ca{} \Cl A{} we have $\f\le\psi$ if $\f(x)\le \psi(x)$ for all positive $x\in \Cl A$.   Our exposition is organized as follows.  In the second section we extend the result of our  previous work in \cite{CHW} by showing that   a $\si$-finite \vNa{} \Cl M {}is finite if and only if  the strong operator topology and  the order topology on the  effect algebra $\Cl E(\Cl M)$  of \Cl M{} have the same null sequences.    In the third section  we demonstrate  that the order topology on the hermitian part of the predual of a \vNa{} is precisely the norm topology.  This is in contrast with the order topology on operator algebras -- in general the later is not linear. Therefore, in dual spaces the order  determines the norm topology.  This explicitly means that if  $T: \Cl A^\ast_s \to \Cl B^\ast_s$ is a bijection (not necessarily linear) between the hermitian parts of the duals of \Ca s $\Cl A$ and $\Cl B$ preserving the order in both directions then $T$ is a homeomorphism  with respect to the norm topologies. Our main results are given in the fourth section.  Let \f{} be a  state  on a \Ca{}  \Cl A{}. We investigate the order topology of the interval $[0,\f]$ and it is shown that this topology differs from the subspace topology induced by the order topology $\tau_o(\Cl A^\ast_s)$.  Indeed, the norm null sequences coincide with the null sequences in the order topology $\tau_o[0,\f]$ exactly when the commutant, $\pi_\f(\Cl A)'$,  of the GNS algebra is a finite \vNa. Using modular theory we then show that a \si-finite \vNa{} is finite if and only if
the null sequences with respect to the norm and the order topology $\tau_o[0,\f]$ coincide for some (every) normal faithful state \f{} on \Cl M. Pursuing the matter  further we show that on \Ca s that are not of Type I we can always find an interval $[0,\f]$ in the dual space on which  convergence to zero in norm and order topology are not  equivalent.     In contrast to this  we show that  a \Ca{} \Cl A{} is abelian if and only if for each state \f{} on \Cl A{} the norm and the order topology on $[0, \f]$ coincide. One one hand this extends classical results on order topology on measurable functions. On the other hand  it characterizes abelian \Ca{} in terms of order topology. As an important  technical  tool we are using the map that identifies positive functionals dominated by a given state with operators in the effect algebra of the commutant algebra associated with the GNS representation. This can be viewed as an analog of the Radon-Nikodym Theorem.  We are also giving some analysis of the continuity properties of this map which may be of independent interest.

Let us recall a few notions and fix the notation.
If $\tau_1$ and $\tau_2$ are two topologies on a set $X$ such that $\tau_1\subset \tau_2$ then we say that $\tau_2$ is finer than $\tau_1$.

Let $(P, \le)$ be a poset. Given $a,b\in P$ we shall use the symbol  $[a,b]$ for the set
$\set{x\in P}{a\le x \le b}$. If $S\subset P$, then $\bigvee S$ and $\bigwedge S$ will denote the supremum and the infimum of $S$, respectively (on condition that they exist). Let $(x_\gamma)_{\gamma\in \Gamma}$ be a net in $P$. We shall say that this  net is increasing (resp. decreasing) if $x_{\gamma_1}\le x_{\gamma_2}$ (resp.  $x_{\gamma_1}\ge x_{\gamma_2}$) whenever  $\gamma_1\le \gamma_2$. The symbol
$x_\gamma \uparrow x$ means that the net $(x_\gamma)$ is increasing and $x=\bigvee \set{x_\gamma}{\gamma \in \Gamma}$. Similarly,  $x_\gamma \downarrow x$ means that the net $(x_\gamma)$ is decreasing  and $x=\bigwedge \set{x_\gamma}{\gamma \in \Gamma}$. A net $(x_\gamma)_{\gamma\in \Gamma}$ is said to \emph{order converge} to $x\in P$ if there exist nets
$(y_\gamma)_{\gamma\in \Gamma}$ and $(z_\gamma)_{\gamma\in \Gamma}$  \st{}
\[ y_\gamma \le x_\gamma \le z _\gamma \qquad  \mbox{ and }\qquad y_\gamma\uparrow x, z_\gamma\downarrow x\,.\]
A subset $X\subset P$ is called \emph{order-closed} (resp. \emph{sequentially order-closed}) if there is no net in $X$ (resp. no sequence) order converging to a point outside $X$. The collection of order-closed sets (resp. sequentially order-closed sets) determines a  topology called  the \emph{order topology} (resp. \emph{sequential order topology}) on $P$. By the symbol $\tau_o(P)$ and $\tau_{os}(P)$ we shall mean the order topology  and the sequential order topology on $P$, respectively. Clearly, $\tau_o(P)\subset  \tau_{os}(P)$. We shall say that $P$ is \emph{monotone order separable} if  for every increasing (and resp. decreasing) net $(x_\gamma)_{\gamma\in \Gamma}$ there is  an increasing (resp. decreasing) sequence $(\gamma_n)$ in $\Gamma$ \st{} $\bigvee_{n\in\nn} x_{\gamma_n}=\bigvee_{\gamma\in\Gamma} x_\gamma$ (resp.  $\bigwedge_{n\in\nn} x_{\gamma_n}=\bigwedge_{\gamma\in\Gamma} x_\gamma$). It has been proved in \cite[Proposition 3]{BCW} that $\tau_o(P)=\tau_{os}(P)$ if and only if $P$ is monotone order separable. We shall use many times the following result proved in  \cite[Proposition 2]{BCW}: The sequence $(x_n)$ in $P$ converges in the sequential order topology to   $x\in P$ if and only if from each subsequence of $(x_n)$ we can extract a further subsequence that order converges to $x$.

Let us remark that for a nonempty subset $P_0$ of  $P$, the subspace topology $(P_0, \tau_o(P))$ and  the intrinsic order topology $(P_0, \tau_o(P_0))$ are not comparable in general. This will also follow from the results of the present paper.

  Let $X$ be a Banach space. Then $X_1$ will denote its closed unit ball. By $X^*$ and $X_*$ we shall denote the dual and the predual of $X$, respectively.  For a set $Y\subset X$ we shall denote by $[Y]$ the closed linear span of $Y$.   The norm topology on $X$  will be denoted by $\tau_{\|\cdot\|}(X)$.

  Throughout   the paper let  \Cl A{} be a unital \Ca. For all unmentioned details on operator algebras we refer the reader to the monographs \cite{B,Takesaki,Kad2,Sakai}.      The unit of \Cl A{}  will be denoted by $\mathds 1$.   Let $\Cl A_s=\set{x\in \Cl A}{x=x^*}$ and $\Cl A_+= \set{x^*x}{x\in \Cl A}$.  Then $\Cl A_s$ is a real vector space and $\Cl A_+$ is a cone in it.  When equipped with the partial order $\le$ induced by the cone $\Cl A_+$ the self-adjoint elements $\Cl A_s$ form an ordered vector space with order unit $\mathds{1}$.  The \emph{effect algebra} of $\Cl A$ is defined  as $\Cl E(\Cl A)=\set{x\in \Cl A}{0\le x \le \mathds 1}$. Let us remark that  the term \emph{effect} comes from an operator algebraic approach to quantum theory, where the structure of  positive contractive operators is connected with quantum measurement.      The projection poset $\Cl P(\Cl A)$ is the set $\set{p\in \Cl A}{p=p^2=p^*}$ equipped with the partial order inherited from $\Cl A_s$. Let $\Cl A^*_s$ be the set $\set{\f\in \Cl A^*}{\f(x)\in \rr \mbox{ for all }x=x^*}$. A positive functional $\f$ on \Cl A{} is a functional such that $\f(x)\ge 0$ for every $x\ge 0$ in \Cl A.  A  positive functional having unit norm is called a state. A positive functional \f{} is called faithful if $\f(a^*a)=0$ implies that $a=0$. The set of positive functionals on $\Cl A$, denoted by $\Cl A_+^*$, is a cone in $\Cl A^*_s$ and thus it induces a partial order on $\Cl A^*_s$. The symbol $B(H)$ will denote the \Ca{} of all bounded operators acting on a \HS{} $H$.  For a set $X\subset B(H)$ we shall denote its commutant by $X'= \set{y\in B(H)}{yx=xy \mbox{ for all } x\in X}$.
Given $\xi\in H$, the vector functional $\omega_\xi$ will be  defined by $\omega_\xi(a)= \langle a\xi, \xi \rangle$. It is a positive functional when restricted to any \Csa{}   of $B(H)$.
Let  $\Cl A\subset B(H)$ and   $\xi$ in $H$. The vector $\xi$ is called separating for \Cl A{} if
$a\xi=0$ implies $a=0$ for all $a\in \Cl A$. The vector $\xi$ is called cyclic  for \Cl A{} if $[\Cl A\xi]=H$. Finally, $\xi$ is said to be bicyclic for $\Cl A$ if it is both cyclic and separating.
Given a \Ca{} \Cl A{} we shall denote by $M_k(\Cl A)$ the matrix algebra of all $k\times k$ matrices  over \Cl A.

 Given a positive functional \f{} on \Cl A{},   there is a Hilbert space $H_\f$, a  $*$-representation $\pi_\f:\Cl A\to B(H_\f)$,  and a vector $\xi_\f\in H_\f$  such that $\f(a)=\langle \pi_\f(a)\xi_\f, \xi_\f\rangle$ ($a\in \Cl A$) and $[\pi_\f(\Cl A)\xi_\f]=H_\f.$  The triple $(H_\f, \pi_\f, \xi_\f)$  is  called the GNS representation of \f.  It is unique up to unitary equivalence.
A \Ca{} \Cl A{}  is called a Type I \Ca{} if each nonzero quotient \Cl B{} of \Cl A{} contains a nonzero positive element $x$ such that the closure of $x\Cl B x$ is an abelian \Ca.

A von Neumann algebra is defined as a \Ca{} that is a dual  Banach space. Throughout  the paper \Cl M{} will represent a \vNa.  Elements of the  predual of $\Cl M$ are called normal functionals.  A \Ca{} $\Cl A\subset B(H)$     containing the identity operator on $H$ is a \vNa{} if and only if $\Cl A = \Cl A''$. If \f{} is a normal positive  functional on \Cl M{} and $\psi$ is a positive functional with $\psi\le \f$, then $\psi$ is normal as well.
 For a functional \f{} in $\Cl M_*$  we shall denote by $|\f|$ its absolute value, that is $|\f|$ is a positive functional in $\Cl M_*$ such that  it has the same norm as $\f$ and satisfies $|\f(x)|^2\le \|{\f}\| \cdot |\f|(xx^*)$ for all $x\in \Cl M$. In the case when \f{} is hermitian we have that $\f\le |\f|$.    Let  us recall basic topologies associated with \vNa s.  The $\si$-strong topology on \Cl M{} is  the locally convex topology generated by the seminorms $\set{ \ro_\psi}{\psi \in \Cl M_*^+}$, where $\ro_\psi(x)= \sqrt{ \psi(x^*x)}$. It will be denoted by $s(\Cl M,\Cl M_*)$.  If \Cl M{} acts on a \HS{} $H$, we denote by  $\tau_s(\Cl M)$ the strong operator topology, i.e. the topology determined by the seminorms $\set{p_\xi}{\xi\in H}$, where $p_\xi(x)=\|x\xi\|$.  We have $\tau_s(\Cl M) \subset \si(\Cl M,\Cl M_*)$. These  two topologies  coincide on bounded subsets of $B(H)$.
A \vNa{} is called $\si$-finite if it does not admit any uncountable system of pairwise orthogonal nonzero projections. A state \f{} on a \Ca{} is called a trace if $\f(x^*x)=\f(xx^*)$ for all $x\in \Cl A$. For a normal trace $\psi$  on a \vNa{} \Cl M{} we have the subadditivity property with respect to projections: $\psi(\bigvee_n p_n)\le \sum_n \psi(p_n)$ where $(p_n)$ is a sequence of projections and the supremum is computed in the projection lattice. A \Ca{} \Cl A{} is called  finite if $aa^*=\mathds 1$ whenever $a\in\Cl A$ and $a^*a=\mathds 1$. A projection $p$ in \Cl A{} is called finite if the algebra $p\Cl A p$ is finite.  $\Cl M$ is called Type II if it finite and contains no nonzero projection $p$  such that  $p\Cl A p$ is abelian.  $\Cl M$ is called Type III if it does not contain any nonzero finite projection.

\section{The order topology on the effects}

In this section we  consider the order topology on  the effect structure $\Cl E(\Cl M)$ and prove a theorem that gives a new topological characterization of finite \vNa s.
This theorem is a key result for our further investigation presented in this  paper.  Although the proof is a variant of  \cite[Theorem 5.3]{CHW} we prefer to give a self-contained argument for the sake of completeness.

\begin{theorem}\label{1.0}
Let \Cl M{}  be a \si-finite \vNa. Then the  following statements are equivalent.
\begin{enumerate}
\item \Cl M{} is finite.
\item Every sequence $(x_n)$ in $\Cl E(\Cl M)$ converging $\si$-strongly to zero converges to zero with respect to $\tau_o(\Cl E(\Cl M))$.
\end{enumerate}
\end{theorem}
\begin{proof}
$(i) \Rightarrow (ii)$. Let $(x_n)_{n\in\nn}$ be a sequence in $\Cl E(\Cl M)$ converging $\sigma$-strongly to $0$.  To prove the  convergence in the order topology  we need  to find a subsequence $(x_{n_i})_{i\in\nn}$ of $(x_n)_{n\in\nn}$ that order converges to $0$.  Since $\Cl M$ is finite and $\sigma$-finite, $\Cl M$ admits a faithful, normal, tracial state,  $\psi$.   By the assumption we can  extract a subsequence $(x_{n_i})_{i\in\nn}$ of $(x_n)_{n\in\nn}$ such that \[\psi(x_{n_i})\le\sqrt{\psi(x_{n_i}^2)}<4^{-i}\,.\]
  For each $i\in\nn$ and $\lambda\in\rr$, let $e^i(\lambda)$ be the spectral projection in $\Cl M$, of $x_{n_i}$,  corresponding to $\lambda$.
  (That is,   $e^i(\lambda)=\chi_{(-\8, \l]}(x_{n_i})$.)
  Then
\begin{align*}
0\le x_{n_i}&\le 2^{-i}\,e^i(2^{-i})+(\mathds 1-e^i(2^{-i}))\\
&=2^{-i}\biggl(\bigwedge_{j\ge i}e^j(2^{-j})+e^i(2^{-i})-\bigwedge_{j\ge i}e^j(2^{-j})\biggr)+(\mathds 1-e^i(2^{-i}))\\
&\le 2^{-i}\bigwedge_{j\ge i} e^j(2^{-j})+e^i(2^{-i})-\bigwedge_{j\ge i} e^j(2^{-j})+\mathds 1-e^i(2^{-i})\\
&=2^{-i}\bigwedge_{j\ge i} e^j(2^{-j})+\bigvee_{j\ge i}\bigl(\mathds 1-e^{j}(2^{-j})\bigr).
\end{align*}
Let
\[y_i=2^{-i}\bigwedge_{j\ge i} e^j(2^{-j})+\bigvee_{j\ge i}\bigl(\mathds 1-e^{j}(2^{-j})\bigr).\]
Then $0\le x_{n_i}\le y_i\le\mathds 1$.
    Let us verify that the sequence $(y_i)_{i\in\nn}$ is decreasing (see also Lemma 5.2 in \cite{CHW}).   Indeed,

\begin{align*}
y_i-y_{i+1}\ =\ &2^{-i}\bigwedge_{j\ge i}e^j(2^{-j})\ +\ \bigvee_{j\ge i}\bigl(\mathds 1-e^{j}(2^{-j})\bigr)\\
&\qquad\qquad\qquad \ \ -\  2^{-i-1}\bigwedge_{j\ge i+1} e^j(2^{-j})\ -\ \bigvee_{j\ge i+1}\bigl(\mathds 1-e^{j}(2^{-j})\bigr)\\
\ge\  &2^{-i}\bigwedge_{j\ge i}e^j(2^{-j})\ -\ 2^{-i-1}\bigwedge_{j\ge i+1}e^j(2^{-j})\\
=\  &2^{-i}\biggl(\bigwedge_{j\ge i}e^j(2^{-j})\ -\ \bigwedge_{j\ge i+1}e^j(2^{-j}\biggr)\ +\ (2^{-i}-2^{-i-1})\bigwedge_{j\ge i+1}e^j(2^{-j})\\
\ge\  &0\,.
\end{align*}

 Thus, $\bigwedge_{i\in\nn}y_i$ exists in $\Cl E(\Cl M)$.  The normality of $\psi$ entails that $\psi\bigl(\bigwedge_{i\in\nn}y_i\bigr)=\lim_{i\to\infty}\psi(y_i)$. Since $2^{-j}\bigl(\mathds 1-e^j(2^{-j})\bigr)\le x_{n_j}$, it follows that $\psi\bigl(\mathds 1-e^j(2^{-j})\bigr)\le 2^j\psi(x_{n_j})<2^{-j}$.

Since $\psi$ is $\sigma$-subadditive we can estimate:
\[\psi(y_i)\le 2^{-i}+\sum_{j\ge i}\psi\bigl(\mathds 1-e^j(2^{-j})\bigr)<2^{-i}+\sum_{j\ge i}2^{-j}=3\cdot 2^{-i}.\]
Thus, $\psi\bigl(\bigwedge_{i\in\nn}y_i\bigr)=0$ and therefore, since $\psi$ is faithful, it follows that $\bigwedge_{i\in\nn}y_i=0$. Consequently,  $(x_{n_i})_{i\in\nn}$ is order-convergent to $0$.

 $(ii) \Rightarrow (i)$. If \Cl M{} is not finite then it has a properly infinite direct summand $\Cl N$. Furthermore,  by the structure theory, \Cl N{}  contains a unital von Neumann subalgebra \Cl R{} that is isomorphic  to $B(H)$ for some infinite-dimensional separable \HS\ $H$.  We shall identify $\Cl R$ with $B(H)$.  Fix an orthonormal basis $(\xi_n)$ of $H$ and let $p_n$ be  the  projection of $H$ onto
$[\{ \xi_n+\frac 1n \xi_1, \xi_{n+1}, \xi_{n+2},\ldots \}]$. For  every $\eta\in H$ let us denote by $\al_n=\langle \eta, \xi_n\rangle$ the coordinates of $\eta$ with respect to  $(\xi_n)$. Then,  an easy computation gives
\[ \nor{p_n\eta}^2=  \frac{|\frac 1n\al_1+\al_n|^2}{\frac 1 {n^2}+1} + \sum_{i=n+1}^\8 |\al_i|^2\,,\]
showing that the sequence $(p_n)$ goes to zero in the strong operator topology and therefore in $s(\Cl R,\Cl R_*)$.  Since the restriction to \Cl R{} of the $\sigma$-strong topology $s(\Cl M,\Cl M_*)$ is equal to $s(\Cl R,\Cl R_*)$, it follows that $(p_n)$ is a  sequence in $\Cl E(\Cl M)$ that is null with respect to $s(\Cl M,\Cl M_*)$.  Let us show that  this sequence cannot be a null sequence with respect to the topology $\tau_o(\Cl E(\Cl M))$. Suppose, for a contradiction, that there is a subsequence $(p_{n_i})$ that is order converging to zero. Let $(a_n)$ be a decreasing sequence  in $\Cl E(\Cl M)$ with zero infimum that is  witnessing this fact.  Denote by $q$ the projection of $H$ onto $[\{\xi_1\}]$.  Then
 \[ p_{n_i}\le a_i\qquad\text{for every }i\in\nn\]
 implies that $a_i\ge p_{n_j}$ for all $j\ge i$ and therefore, in view of   \cite[Lemma 2.8]{CHW}, we have
 \[ a_i\ge \bigvee_{j\ge i} p_{n_j}\ge q\,,\]
i.e. the infimum of the sequence $(a_n)$ cannot be zero - a contradiction.
\end{proof}

\section{Order topology on dual spaces}

In this section we show that the order topology on the dual space is equal to the norm topology.

If $\Cl K\subset \Cl A^*_s$ is order bounded, i.e. there exists $\psi\in \Cl A^*_+$ such that
\[ -\psi\le\f\le\psi\,,\]
for every $\f\in \Cl K$; then
\[\Vert\psi\Vert\ge\psi(x_+)\ge\f(x_+)\ge\f(x)\ge-\f(x_-)\ge-\psi(x_-)\ge -\Vert\psi\Vert\]
for every $x\in \Cl A_s$ satisfying $\Vert x\Vert \le 1$ and for every $\f\in\Cl K$. ($x_+$ and   $x_-$ denotes positive and negative part of $x$, respectively.) Thus, $\Cl K$ is norm bounded.  On the other-hand, if $\Cl A$ is infinite-dimensional, it contains a sequence $(a_n)$ of  positive elements satisfying $\Vert a_i a_j\Vert=\delta_{ij}$.  So, if $\f_i$ is a state on $\Cl A$ satisfying $\f_i(a_j)=\delta_{ij}$, it is easily seen that $\{\f_i:i\in\nn\}$ is a norm bounded subset of $\Cl A^*_+$ that is not order bounded.

\begin{lemma}\label{1.11}
Let $(\f_\al)$ be a monotonic increasing and norm bounded net  of hermitian elements in $\Cl M_\ast$. Then there is a unique normal hermitian functional \f{} on $\Cl M$ \st{}
  \begin{equation}\label{haj} \f(x)=\sup_\al  \f_\al(x)
  \end{equation}
for all positive elements $x\in \Cl M$. Moreover, $\f$ is the norm limit of the net $(\f_\al)$.
\end{lemma}
\begin{proof}
The net  of real numbers $(\f_\al(x))$ is increasing and bounded for every $x\in\Cl M_+$.
Since $\Cl M$ is linearly generated by its positive elements,  it follows that there is a unique linear form \f{} on $\Cl M$ satisfying (\ref{haj}).
Moreover, \f{} is bounded on  the positive part of the unit ball and so it is a bounded functional. Using the fact that $\f-\f_\al$ is positive for each $\al$ we have that
\[ \nor{\f-\f_\al}= \f(\mathds 1)-\f_\al(\mathds 1)\,,\]
which goes to zero. Therefore, $\f$  is the norm limit of $(\f_\al)$. Employing  the fact that the space of normal functionals is norm closed,  we conclude that $\f$ is a normal functional.


\end{proof}

Clearly, the dual assertion holds for a monotonic decreasing net.  We recall that a poset $(P,\le)$ is \emph{conditional monotone complete} if every monotonic increasing net (or monotonic decreasing net) having an upper bound (resp. a lower bound) has a supremum (resp. an infimum).  Thus, the lemma implies  that $(M_\ast^s, \le)$ is conditional monotone complete.


\begin{lemma}\label{1}
Suppose that a net $(\f_\al)$ order converges to $\f$ in $\Cl M_\ast^s$. Then $(\f_\al)$ converges to \f{} in norm.
\end{lemma}
\begin{proof}
Suppose that
\[\psi_\al'\le \f_\al\le \psi_\al\,,\]
in $\Cl M_*^s$ and
 $\psi_\al\downarrow \f$,  $\psi_\al' \uparrow \f$. Then by the previous proposition  $\psi_\al, \psi_\al'\to \f$
in norm. For every $x\in \Cl M_1^+$  and for every $\al$ we have
\[ \psi'_\al(x)\le \f_\al(x)\le \psi_\al(x) \quad\text{and}\quad\psi'_\al(x)\le \f(x)\le \psi_\al(x)\,.\]
This implies that
\[ |\f_\al(x)-\f(x)|\le |\psi_\al(x)-\psi_\al'(x)|\,.\]
Thus,
\[\sup_{x\in \Cl M_1^+}  |\f_\al(x)-\f(x)|\le \sup_{x\in \Cl M_1^+} |\psi_\al(x)-\psi_\al'(x)|\to 0\quad\text{with }\alpha.\]
Since every element in  the unit ball can be written as $x_1-x_2+i(x_3-x_4)$, where $x_i\in \Cl E(\Cl M)$ ($i=1,2,3,4$) it follows that $\f$ is the norm limit of $(\f_\al)$.
\end{proof}


The function $f$  mapping a poset $(P,\le)$ into another poset $(Q,\le)$ is \emph{order-continuous} if $(f(x_\gamma))$ order converges to $f(x)$ whenever $(x_\gamma)$ is a net in $P$ that is order convergent to $x$.  It is easy to verify that if $f$ is order-continuous, then $f$ is continuous with respect to $\tau_o(P)$ and $\tau_o(Q)$.  When $f(a)\le f(b)$ for every $a,b\in P$ satisfying $a\le b$, we say that $f$ is \emph{order-preserving} or \emph{isotone}.

\begin{proposition}\label{1.12} Let $\Cl M$ and $\Cl N$  be  von Neumann algebras, and let $A$ and $B$ be closed subsets of $\Cl{M}_\ast^s$ and $\Cl N_\ast^s$, respectively.  The following assertions hold.
\begin{enumerate}
  \item $\tau_o(\Cl{M}_\ast^s)|A\subset \tau_o(A)$
  \item If $f:A\to B$ is order-preserving and norm-continuous, then it is continuous with respect to $\tau_o(A)$ and $\tau_o(B)$.
\end{enumerate}
\end{proposition}
\begin{proof}
By Lemma \ref{1} it follows that $A$ is closed with respect to $\tau_o(\Cl M_\ast^s)$.  Since $\Cl M_\ast^s$ is conditional monotone complete, the first assertion follows by \cite[Proposition 2.3]{CHW}.  For the second assertion it suffices to show that if $\f_\gamma\uparrow \f$  and $\psi_\gamma\downarrow \f$ in $A$, then $f(\f_\gamma)\uparrow f(\f)$ and $f(\psi_\gamma)\downarrow f(\f)$ in $B$.  We shall prove the assertion for $(\f_\gamma)$; the other follows by a dual argument. If $\f_\gamma$ is monotonic increasing and bounded above by $\f$ (in particular $\f_\gamma$ is norm-bounded) then it follows, by Lemma \ref{1.11}, that $(\f_\gamma)$ is norm-convergent to its supremum taken in $\Cl{M}_*^s$.  Our assumption on the closure of $A$ allows us to deduce that $\f$ must be equal to this supremum and that therefore (again by Lemma \ref{1.11}) $(\f_\gamma)$ is norm-convergent to $\f$.  Since $f$ is order-preserving, it follows that $\bigl(f(\f_\gamma)\bigr)$ is monotonic increasing and bounded above by $f(\f)$.  By Lemma \ref{1.11}, we know that $\bigl(f(\f_\gamma)\bigr)$ is norm-convergent to its supremum taken in $\Cl N_\ast^s$.  Our assumption on the continuity of $f$ (and the Hausdorffness of the norm-topology) gives us that this supremum must be equal to $f(\f)$.  Since $B$ is assumed to be closed, $f(\f)$ must be in $B$ and thus, $f(\f)$ must be equal to the supremum of
$\bigl(f(\f_\gamma)\bigr)$ taken in $B$.
\end{proof}

\begin{theorem}\label{1.1} Let \Cl M{} be \vNa. Then
\[ \tau_{\nor{\cdot }}(\Cl M_*^s)= \tau_o(\Cl M_\ast^s)= \tau_{os}(\Cl M_\ast^s)\,.\]
In particular, $\Cl M_*^s$ is monotone order separable.
\end{theorem}
\begin{proof}
By Lemma~\ref{1} we know that the order topology $\tau_o(\Cl M_\ast^s)$ is finer than the norm topology.  We show that $\tau_{\nor{\cdot }}(\Cl M)\supset \tau_{os}(\Cl M_\ast^s)$.  In fact, we show a bit more.  We prove that if $(\f_n)$
 is a sequence in $\Cl M_*^s$ that is norm converging to \f{}, then it converges to $\f$ with respect to $\tau_{os}(\Cl M_\ast^s)$.  To this end, we need to show that $(\f_n)$ contains a subsequence that order converges
to \f. But passing to a suitable subsequence and making a little abuse of notation,  we can suppose that
\[ \nor{\f_n-\f}\le 2^{-n}\qquad\text{for every }n\in\nn\,.\]
Put 
\[ v_i=\sum_{n\ge i} |\f_n-\f|\,.\]
Clearly,
\[-v_n\le -|\f_n-\f| \le\f_n-\f \le |\f_n-\f|\le v_n\]
and so
\[ \f-v_n\le \f_n\le \f+ v_n\,. \]
Since $v_n\downarrow 0$ we have $\f+v_n\downarrow \f$ and $\f-v_n\uparrow \f$.  Thus $(\f_n)$ order converges to $\f$.  That $\Cl M_*^s$ is monotone order separable, follows from \cite[Proposition 3]{BCW}.
\end{proof}

The dual space of a \Ca{} is a special case of a predual  of a \vNa{}.  Therefore,  the norm topology coincides with the order topology on the hermitian part of the dual of every \Ca.

\begin{corollary}
Let \Cl A{} be a \Ca. Then
\[ \tau_{\nor{\cdot }}(\Cl A)=  \tau_o(\Cl A^\ast_s)=\tau_{os}(\Cl A^\ast_s)\,,\]
and $\Cl A^*_s$ is monotone order separable.
\end{corollary}

\section{Order topology on  bounded parts of dual spaces}

In this section as the underlying posets we consider intervals of $\Cl M_\ast^s$.   Let us note that  in $\Cl M_\ast^s$, translation by a fixed element is an order-isomorphism.  Therefore, given $\f$ and $\psi$ in $\Cl M_\ast^s$,   the topological spaces   $\bigl([\f,\psi], \tau_o[\f,\psi]\bigr)$ and $\bigl([0, \f-\psi], \tau_o[0,\f-\psi]\bigr)$ are homeomorphic. That is why we consider only intervals of type $[0, \f]$ for $\f\in M_\ast^+$. Let us recall the easily verified fact that every closed interval in a poset is order-closed.

We shall start by showing in Theorem \ref{2.0} that the condition $\tau_o(\Cl M_\ast^s)|[0,\f]=\tau_o[0,\f]$ for every $\f\in \Cl M_\ast^+$ is equivalent to the commutativity of $\Cl M$.  First we prove the following lemma.   We recall that via the duality $(a,b)\mapsto\tr(ab)$ the predual of the algebra of $m\times m$ complex matrices $M_m$ can be identified (as a Banach space) with $(M_m,\Vert\cdot\Vert_1)$ where $\Vert a\Vert_1=\tr(|a|)$.  If we denote by $\f_a$ the linear functional on $M_m$ associated with $a\in M_m$ (i.e. $\f_a:b\mapsto\tr(ab)$) then $a\ge 0$ if and only if $\f_a\ge 0$.  It follows therefore that the interval $[0,\mathds{1}]$ (in $M_m$) is order-isomorphic to $[0,\f_\mathds{1}]$ (in $M_m^\ast$).

 \begin{lemma}\label{lem3}
Let $\f_\mathds{1}$ denote the trace function on the algebra of $m\times m$ complex matrices $M_m$ (i.e.  $\f_{\mathds{1}}:a\mapsto\tr(a)$). Then the restriction of the norm topology to the interval $[0,\f_{\mathds{1}}]$ is not equal to  $\tau_o[0, \f_{\mathds{1}}]$.
\end{lemma}
\begin{proof}
Let $\{\xi_1,\dots,\xi_m\}$ denote the canonical basis in $\cc^m$ and define
\[ \eta_n= \cos\theta_n\, \xi_1 + \sin \theta_n\, \xi_2\,,  \]
    where $(\theta_n)$ is a strictly increasing  sequence in $(\pi/4, \pi/2)$ with limit $\pi/2$.  Then $(\eta_n)$ converges to $\xi_2$ and any two members of this sequence are linearly independent.  Denote by $p_n$ the projection of $\cc^m$ onto $\eta_n$, by $p$ the projection onto $\xi_2$, and by $q$ the projection onto $[\{\xi_1,\xi_2\}]$.
 The linear functionals $\f_{p_n}:a\mapsto\tr(p_na)$ are in $[0, \f_{\mathds{1}}]$
and converge in norm to the functional $\f_p:a\mapsto\tr(pa)$. We shall show, however that $(\f_{p_n})$ does not converge to $\f_p$ with respect $\tau_0[0,\f_{\mathds{1}}]$.  In the light of \cite[ Proposition 2.1]{CHW} it suffices to show that no subsequence of $(\f_{p_n})$ is order-convergent to $\f_p$.  Since $\Cl{E}(M_m)$ is order-isomorphic with $[0,\f_{\mathds{1}}]$  via the association $a\mapsto \f_a$, our goal is achieved when we show that no subsequence of $(p_n)$ is order-convergent to $p$ in $\Cl{E}(M_m)$.  This latter assertion follows by \cite[Lemma 2.8 (ii)]{CHW} because if $\mathds 1\ge x$ satisfies $x\ge p_i$ and $x\ge p_j$ for some $i\neq j$, then $x\ge q$.  So, if it exists, the order limit in $\Cl{E}(M_m)$ of any subsequence of $(p_n)$ must dominate $q$.






\end{proof}

\begin{theorem}\label{2.0}
For a von Neumann algebra $\Cl{M}$ the following statements are equivalent.
\begin{enumerate}[{\rm(i)}]
  \item $\Cl M$ is abelian.
  \item For every $\f\in\Cl{M}_*^+$ the restriction of the norm-topology to $[0,\f]$ is equal to $\tau_o[0,\f]$.
\end{enumerate}
\end{theorem}
\begin{proof}
${\rm(i)}\Rightarrow{\rm(ii)}$  If $\Cl{M}$ is abelian, then $\Cl M_\ast^s$ is isometrically isomorphic  (as an ordered normed space) with $L^1(X,\Sigma,\mu,\rr)$ for some localisable measure space $(X,\Sigma,\mu)$.  So it suffices to show that if $u\in L^1(X,\Sigma,\mu)$ and $u\ge 0$, then the $L^1$-norm topology on $[0,u]$ coincides with the order topology $\tau_o[0,u]$.  Since $L^1(X,\Sigma,\mu,\rr)$ is a Dedekind complete lattice (see for example \cite[244L p. 167]{Fremlin}), and since the interval $[0,u]$ is an order-closed sublattice of $L^1(X,\Sigma,\mu,\rr)$, the final assertion follows by \cite[Proposition 2.3 (ii)]{CHW} and by Theorem \ref{1.1}.

${\rm(ii)}\Rightarrow{\rm(i)}$ If $\Cl M$ is not abelian, then there exists a nonzero central projection $z$  and an integer $m\ge 2$ \st{} the  direct summand $z\Cl M$ contains a unital subalgebra $\Cl N$ that is $\ast$-isomorphic to $M_m$.
Let us first suppose that $z=\mathds 1$.  Let $\tilde\f_{\mathds{1}}$ denote the trace function on $\Cl N$, and let $\tilde\f_n$ $(n\in\nn)$ and $\tilde\f$ denote the positive linear functionals on $\Cl N$ produced in the proof of Lemma~\ref{lem3}, i.e.
\begin{enumerate}
\item $\tilde{\f}_n$ and $\tilde{\f}$ belong to $[0,\tilde{\f}_{\mathds{1}}]$,
\item $(\tilde{\f}_n)$ goes in norm to  $\tilde{\f}$,
\item  $(\tilde{\f}_n)$ is not convergent with respect to  $\tau_o[0, \tilde{\f}_{\mathds{1}}]$.
\end{enumerate}

By \cite[IV 2.2.4, p. 355]{Blackadar}) we know that there is a normal conditional expectation $\Phi$ of $\Cl M$ onto $\Cl N$.    For a general  functional $\tilde{\omega}$ on   $\Cl N$   denote by ${\omega}$ the normal functional on $\Cl M$ given by
  \[ {\omega}= \tilde{\omega} \circ \Phi\,.\]
By the properties of normal conditional expectation it follows that the map $\tilde{\omega}\to {\omega}$ is order and norm preserving.  Thus, $\f_{\mathds{1}}$ is a positive linear functional on $\Cl M$, the $\f_n$'s and $\f$ belong to $[0,\f_{\mathds 1}]$, and $\Vert\f_n-\f\Vert\to 0$.

The function $f:[0,\f_{\mathds{1}}]\to[0,\tilde{\f}_{\mathds{1}}]$ defined by $f(\f)=\f|_N$ is order-preserving and norm continuous.  Moreover, since  $[0,\f_{\mathds{1}}]$ and $[0,\tilde{\f}_{\mathds{1}}]$ are closed with respect to $\tau_o(\Cl M_\ast^s)$ and $\tau_o(\Cl N_\ast^s)$, respectively, {\rm(ii)} of  Proposition \ref{1.12} implies that $f$ is continuous with respect to $\tau_o[0,\f_{\mathds{1}}]$ and $\tau_o[0,\tilde{\f}_{\mathds{1}}]$.  So, if $\f_n\to\f$ with respect to $\tau_o[0,\f_{\mathds{1}}]$, then $\tilde{\f}_n\to\tilde{\f}$ with respect to $\tau_o[0,\tilde{\f}_{\mathds{1}}]$, which is not true.

 Now consider the case when $z\neq\mathds{1}$. Observe that if $\tau$ is a normal positive linear functional on $\Cl M$ \st{} $\tau(z)=\Vert \tau\Vert$ then  any  positive linear functional majorized by $\tau$ is zero on $(\mathds 1-z)\Cl M$.  Therefore the interval $[0, \tau]\subset \Cl M_*$ is order-isomorphic to the same interval taken in $(z\Cl M)_*$.  So the result follows by the previous part of the proof.
\end{proof}

Clearly, a C$^\ast$-algebra $\Cl A$ is abelian if and only if the enveloping von Neumann algebra $\Cl A^{\ast\ast}$ is abelian.  Therefore we get the following corollary.

\begin{corollary}\label{5.1}
For a C$^\ast$-algebra $\Cl A$ the following  two statements are equivalent.
\begin{enumerate}[{\rm(i)}]
\item $\Cl A$ is abelian.
\item For every nonzero $\f\in\Cl{A}^*_+$ the restriction of the norm-topology to $[0,\f]$ is equal to $\tau_o[0,\f]$.
\end{enumerate}
\end{corollary}

We shall now compare the sets of null sequences in intervals of the type $[0,\f]$, $\f\in\Cl A^\ast$, with respect to the norm and order topology $\tau_o[0,\f]$.  For this investigation we shall make an essential use of the properties of the GNS representation and the map $\theta_\f$ that  identifies  the effect algebra of  the commutant of the GNS algebra  with the interval $[0,\f]$ in the dual space. Even though  this  map is important in many aspects of operator algebra theory,  we did not find a proper reference to the continuity properties of this map (and its inverse). We are therefore presenting this analysis that may be of independent interest. Let us fix the notation. \\

For a nonzero positive element \f{}  in the  dual $\Cl A^*$ let $C_\f= \mbox{span} [0,\f]$.  Via the GNS representation of $\Cl A$ induced by $\f$, it is possible to  define a  map
\[ \theta_\f: \pi_\f(\Cl A)'\to \Cl A^*:x\mapsto\theta_\f(x),\]
where
\[\theta_\f(x):\Cl{A}\to \cc:a\mapsto\langle\pi_\f(a)x\xi_\f,\xi_\f \rangle\,. \]
We recall (see for e.g. \cite[Proposition 3.10 p. 201]{Takesaki}) that   $\theta_\f$ is a linear-isomorphism of $\pi_\f(\Cl A)'$ onto $C_\f$ and that if we denote its inverse by $\Phi_\f$ then both $\theta_\f$ and $\Phi_\f$ are completely positive.
In particular, this implies that  the restriction of $\theta_\f$ is   an order-isomorphism between $\Cl E(\pi_\f(\Cl A)')$ and $[0, \f]$.

Let us mention that in  Remark 3.11 (i) on p.202 in \cite{Takesaki} it is erroneously stated that $\theta_\f$ need not be continuous. Although it is true that a completely positive map from a subspace of any of $\Cl A$  into another subspace of   $\Cl A^*$ need not be continuous, the continuity of $\theta_\f$ follows easily by the following estimation:

\[ \|\theta_\f(x)\|\le \|x\|\, \|\xi_\f\|^2\,.\]

On the other hand, it is the inverse map  $\Phi_\f$ that need not be continuous, as the following example shows.

\begin{example}
Let $\zeta=(\alpha_n)$ be a vector in $\ell^1$ such that $\alpha_n>0$ for every $n\in\nn$.  Then $\eta=(\sqrt{\alpha_n})$ belongs to $\ell^2$.    Let $\Cl{A}$ be the $C^*$-algebra $\ell^\infty$ acting on $\ell^2$ and let $\f:=\omega_\eta$ be the positive linear functional  on $\Cl A$ defined by
\[\f:a=(\beta_n)\mapsto\sum_{n=1}^{\infty}\beta_n\alpha_n=\langle a\eta,\eta\rangle\,.\]
Denote by $\ell^p_\zeta$ the sequence space $\mathcal L^p(\nn,\mathscr P(\nn),\mu_\zeta)$ where $\mu_\zeta$ is the $\sigma$-additive measure defined by $\mu_\zeta(\{n\})=\alpha_n$.  Let $\{H_\f,\pi_\f,\xi_\f\}$  be the GNS construction of $\Cl A$ with respect to $\f$.  Then $H_\f=\ell^2_\zeta$; the constant unit sequence equals $\xi_\f$; and $\pi_\f(a)$ equals the operator on $\ell^2_\zeta$ given by multiplication by $a$.  Our assumption on the $\alpha_n$'s implies that $\pi_\f$ is an isomorphism and  therefore  an isometry.  Let $h=(\beta_n)\in \Cl{A}$ and let $\psi_h$ denote the linear functional $\theta_\f(\pi_\f(h))$.  We show that $\Vert \psi_h\Vert =\Vert h\Vert_{\ell^1_\zeta}$.   First observe that for every $a=(\lambda_n)\in\Cl{A}_1$ we have
\[|\psi_h(a)|=|\langle \pi_\f(ah)\xi_\f,\xi_\f\rangle|=\left|\sum_{n=1}^{\infty}\lambda_n\beta_n\alpha_n\right|\le \sum_{n=1}^{\infty}|\beta_n|\alpha_n=\Vert h\Vert_{\ell^1_\zeta}\,.\]
On the other-hand, if we let $u$ denote the unitary element in $\Cl A$ such that $uh=|h|$, then
\[\Vert h\Vert_{\ell^1_\zeta} =\sum_{n=1}^{\infty}|\beta_n|\alpha_n= |\langle \pi_\f(uh)\xi_\f,\xi_\f\rangle|=|\psi_h(u)|\,.\]
Thus, we have that $\Vert \psi_h\Vert =\Vert h\Vert_{\ell^1_\zeta}$ and since the inverse $\Phi_\f$ of $\theta_\f$ maps $\psi_h$ into $\pi_\f(h)$, i.e. $\Phi_\f$ maps $\bigl(\ell^\infty,\Vert\cdot\Vert_{\ell^1_\zeta}\bigr)$ into $\bigl(\ell^\infty,\Vert\cdot\Vert_\infty\bigr)$  it follows that $\Phi_\f$ is not continuous.  Indeed, consider  the sequence $(x_n)$ in $\ell^\infty$,    where $x_n=(\delta_{jn})_j$. We can  see that $\Vert x_n\Vert_\infty=1$, while  $\Vert x_n\Vert_{\ell^1_\zeta}=\al_n\to 0$.     Note that in this case the space $C_\f$ is not complete since by the Open Mapping Theorem this would imply continuity of $\Phi_\f$
\end{example}


\begin{proposition}\label{2.3}
Let $\f\in \Cl A^*_+$.  Then $\Vert \theta_\f(x)\Vert=\Vert x^{\frac{1}{2}}\xi_\f\Vert^2$ for every $x\in \Cl E(\pi_\f(\Cl A)')$.
\end{proposition}
\begin{proof}
For any $x\in \Cl E(\pi_\f(\Cl A)')$ and $a\in \Cl A_1$ we have
\[
  \left|(\theta_\f(x))(a)\right| =\left|\langle \pi_\f(a) x\xi_\f,\xi_\f\rangle\right|=\left|\langle\pi_\f(a)x^{\frac{1}{2}}\xi_\f,x^{\frac{1}{2}}\xi_\f\rangle\right| \le\Vert x^{\frac{1}{2}}\xi_\f\Vert^2\,.\]
To establish the equality one simply needs to consider an approximate identity $(u_\gamma)$ in $\Cl A$.
\[
(\theta_\f(x))(u_\gamma)=\langle \pi_\f(u_\gamma) x\xi_\f,\xi_\f\rangle=\langle \pi_\f(u_\gamma) x\xi_\f-x\xi_\f,\xi_\f\rangle+\langle x\xi_\f,\xi_\f\rangle\ \to\ \Vert x^{\frac{1}{2}}\xi_\f\Vert^2\,.\]
\end{proof}

\begin{corollary}\label{2.31}Let $\f\in \Cl A^*_+$.  A net $(x_\gamma)$ in $\Cl E(\pi_\f(\Cl A)')$ is null with respect to the strong operator topology if and only if $\left(\theta_\f(x_\gamma)\right)$ is null with respect to the norm topology.
\end{corollary}
\begin{proof}
Since $\xi_\f$ is cyclic for $\pi_\f(\Cl A)$, a bounded net $(x_\gamma)$ in $\pi_\f(\Cl A)'$ is null with respect to the strong operator topology if and only if $x_\gamma\xi_\f\to 0$.  Moreover, since  $x_\gamma\ge 0$ for every $\gamma$, the condition $x_\gamma\xi_\f\to0$ is equivalent to $x_\gamma^{\frac{1}{2}}\xi_\f\to 0$.  So the Corollary follows by Proposition \ref{2.3}.
\end{proof}

\begin{theorem}\label{4.1}
Let \f{} be a state on a \Ca{} $\Cl A$. The following statements are equivalent.
\begin{enumerate}[{\rm(i)}]
\item The \vNa{} $\pi_\f(\Cl A)'$ is finite.
\item The interval $[0,\f]$ has the same null sequences with respect to  the norm topology and the order topology $\tau_o[0,\f]$.
\end{enumerate}
\end{theorem}
\begin{proof}
First let us observe that $\pi_\f(A)'$ is $\si$-finite as $\omega_{\xi_\f}$ is a faithful normal state on it. Then the result follows immediately by   Corollary~\ref{2.31} together with Theorem~\ref{1.0}, and  the fact  that $[0,\f]$ and $\Cl E(\pi_\f(\Cl A)')$ are order-isomorphic.
\end{proof}

 Let us remark that when a state \f{} is faithful, then the vector $\xi_\f$ is bicyclic. In this case, deep  modular theory  (see e.g. \cite{B}) tells us  that $\pi_\f(\Cl A)''$ and $\pi_\f(\Cl A)'$  are anti-isomorphic. It follows therefore that  $\pi_\f(\Cl A)''$    is finite if and only if $\pi_\f(\Cl A)'$ is finite (see e.g. \cite[Theorem 9.1.3, p. 588]{Kad2}).
 Therefore we have the following corollary.

We say that a state $\f$ on a $C^*$-algebra  $\Cl{A}$ is a \emph{smooth state} if it satisfies {\rm(ii)} of Theorem \ref{4.1}.

\begin{corollary}
For a faithful state \f{}  on a \Ca{}  $\Cl A$ the following conditions are equivalent:
 \begin{enumerate}
\item   $\pi_\f(\Cl A)''$ is finite,
\item    $\f$ is smooth.
\end{enumerate}
Moreover, if $\Cl M$ is a \vNa{} and \f{} is a faithful normal state on $\Cl M$, then $\f$ is smooth if and only if $\Cl{M}$ is finite.
\end{corollary}
\begin{proof}
The second part follows from  the fact that if \f{} is faithful and normal then $\pi_\f$ is a normal faithful representation and so $\pi_\f(\Cl M)$ is isomorphic as a \vNa{} to $\Cl M$.
\end{proof}

\begin{theorem}\label{3.3} Let $\Cl{A}$ be a $C^*$-algebra.
\begin{enumerate}[{\rm(i)}]
\item If every state  on  $\Cl A$ is smooth then $\Cl{A}$ is of type $I$.
\item In particular, if $\Cl{A}$ is a von Neumann algebra and every normal state on $\Cl A$ is smooth then $\Cl{A}$ is  isomorphic to a finite direct sum of matrix algebras over abelian \vNa s.
\end{enumerate}
\end{theorem}
\begin{proof}
If $\Cl{A}$ is not of type $I$, then by the celebrated Glimm-Sakai Theorem -- established for  the nonseparable case in \cite{Sakai1,Sakai2} -- we know that $\Cl{A}$ has a type $III$ representation, say $(\pi,H)$.  Let $\xi$ be a unit vector in $H$ and let $\f$ denote the state on $\Cl{A}$ defined by $\f(a)=\langle \pi(a)\xi,\xi\rangle$.  Then $H_\f$ is unitarily equivalent to the subspace $H_\xi=[\pi(\Cl{A})\xi]$ of $H$ and the GNS representation $\pi_\f$ is equivalent to the compression of $\pi$ to $H_\xi$.  Denote by $p_\xi$ the projection of $H$ onto $H_\xi$ and let
$\Cl M= \pi_\f(\Cl A)''$.
Since \Cl M{} is of type III,  the same holds for its commutant $\Cl M'$ (see e.g. \cite[Theorem 9.1.3, p. 558]{Kad2}). This property passes  to  the hereditary subalgebra $p_\xi\,  \Cl M'\ p_\xi$ as well. Working on the Hilbert space $H_\xi$ we have
\[ (p_\xi\Cl M)' = (p_\xi\Cl M p_\xi)'=  p_\xi\Cl M' p_\xi\,.\]
So $(p_\xi\Cl M)'$ -- and therefore the algebra $p_\xi\Cl M$ (by \cite[Theorem 9.1.3, p. 558]{Kad2}) --  is of type III.  Since $p_\xi\Cl M$ is isomorphic to $\pi_\f(\Cl{A})''$ it follows that $\pi_\f(\Cl{A})'$ is of type $III$ and therefore infinite.  Hence $\f$ is not smooth by Theorem \ref{4.1}.

The second assertion follows from the fact that any type I \Ca{} is nuclear (see e.g. \cite[Proposition 2.7.3]{Ozawa}). Moreover,  it is known that a \vNa{} \Cl A{} is nuclear if and only if it is a finite direct sum of finite homogeneous algebras, that is exactly when $\Cl A = M_{k_1}(\Cl A_1)\oplus M_{k_2}(\Cl A_2)\oplus \cdots \oplus M_{k_j}(\Cl A_j)$, where $\Cl A_1, \Cl A_2, \ldots \Cl A_j$ are abelian \vNa s.
\end{proof}

{\bf Acknowledgment} The work of Jan Hamhalter was supported by the project  CZ.02.1.01/0.0/0.0/16\_019/00007 and by the project of ``Czech Science Foundation"   [grant number 17-00941S,  ``Topological and geometrical properties of Banach spaces and operator algebras II"].

\end{document}